\documentclass[12pt,reqno]{amsproc}
\usepackage[margin=1in]{geometry}
\usepackage{comment}
\usepackage[small,bf]{caption}
\usepackage{nicefrac}
\usepackage[usenames, dvipsnames]{color}
\usepackage{hyperref}
\hypersetup
{
    colorlinks,	%
    citecolor=blue,%
    filecolor=black,%
    linkcolor=blue,%
    urlcolor=blue
}
\usepackage{amssymb,amsmath,amsthm,amscd}
\usepackage{eulervm, xfrac}
\usepackage[sans]{dsfont}
\usepackage{mathdots}
\usepackage{mathrsfs}
\usepackage{rotating, setspace}
\DeclareMathAlphabet{\mathbfsf}{\encodingdefault}{\sfdefault}{bx}{n}
\usepackage[all]{xy}
\usepackage{enumerate}
\DeclareFontShape{OT1}{cmr}{bx}{sc}{<-> cmbcsc10}{}
\theoremstyle{definition}
\newtheorem{thm}{Theorem}[section]
\newtheorem{lem}[thm]{Lemma}
\newtheorem{defn}[thm]{Definition}
\newtheorem{cor}[thm]{Corollary}
\newtheorem{prop}[thm]{Proposition}

\newtheorem{ex}[thm]{Example}

\newtheorem{rmk}[thm]{Remark}

\newcommand\define[1]{\textbf{#1}}

\newcommand{\calB}{\mathcal{B}}

\newcommand{\Ccat}{\mathbfsf{C}}

\newcommand{\Vect}{\mathbfsf{Vect}}
\newcommand{\vect}{\mathbfsf{vect}}

\newcommand{\Pers}{\mathbfsf{Pers}}

\newcommand{\Top}{\mathbfsf{Top}}
\newcommand{\Set}{\mathbfsf{Set}}
\newcommand{\set}{\mathbfsf{set}}

\newcommand{\Ord}{\mathbfsf{Ord}}

\newcommand{\cat}{\mathbfsf{C}}

\newcommand{\R}{\mathbb{R}}

\newcommand{\field}{\Bbbk}

\newcommand{\id}{{\text{id}}}

\newcommand{\supp}{\text{supp}}

\newcommand{\barcode}{B}
\newcommand{\Ecode}{E}
\newcommand{\BCspace}{\mathcal{B}}

\newcommand{\MTspace}{\mathcal{T}}
\newcommand{\CMTspace}{\mathcal{X}}

\title[The Fiber of the Persistence Map]{The Fiber of the Persistence Map \\
for Functions on the Interval}
\author{Justin Curry}

\begin{document}
\begin{abstract}
In this paper we study functions on the interval that have the same persistent homology, which is what we mean by the fiber of the persistence map.
By imposing an equivalence relation called graph-equivalence, the fiber of the persistence map becomes finite and a precise enumeration is given.
Graph-equivalence classes are indexed by chiral merge trees, which are binary merge trees where a left-right ordering of the children of each vertex is given.
Enumeration of merge trees and chiral merge trees with the same persistence makes essential use of the Elder Rule, which is given its first detailed proof in this paper.
\end{abstract}

\maketitle


\section{Introduction and Acknowledgements}

Let $f:X \to \R$ be a piece-wise linear function on a finite simplicial complex.
Persistence is a new type of geometry that generalizes Morse theory by quantifying the lifetimes of homological features of $X$, when filtered by sub-level sets of $f$.
The lifetime of a homology class is captured using an interval in $\R$ and the collection of the lifetimes of all homology classes is captured using a collection of intervals $\barcode=\{I_j\}$ called the \define{barcode} $\barcode$ or the \define{persistence diagram} associated to $f$---the latter being described in terms of a configuration of points in the extended plane $\mathbb{E}^2$.
The persistence map is the map that takes functions on $X$ to their associated persistence diagrams.

Although persistence has proved remarkably useful in data science~\cite{carlsson2009topology,edelsbrunner2010computational,ghrist-barcodes}, the analytical study of the persistence of functions has received little attention.
We mention two candidate problems in this vein.
\begin{enumerate}
	\item The Realization Problem: What is the image of the persistence map for each topological space $X$ and particular type of function $f$?
	\item The Closed Formula Problem: Is it possible to determine the persistence diagram associated to a semi-algebraic set that is filtered by a semi-algebraic function purely in terms of the equations involved?
\end{enumerate}
In this paper we take up the first question in the simplest possible case---functions on the interval---and go further by giving a complete characterization of the fiber of the persistence map.
Our main result, Theorem~\ref{thm:count-persistence}, proves that, after imposing a graph-equivalence relation on functions with local minima at $x=0$ and $x=1$, the fiber of the persistence map is finite and given by the formula
\[
| PH_0^{-1}(\barcode)| = 2^{N-1}\prod_{j=2}^N\mu_{\barcode}(I_j),
\]
where $\mu_{\barcode}(I_j):=|\{I_k\in\barcode \mid I_j \subset I_k\}|.$

We note that our graph-equivalence relation is a restricted version of topological conjugacy, so in some respects this work is similar to Arnold's Calculus of Snakes~\cite{arnold1992calculus}.
However, unlike Arnold's work, the quantities $\mu_{\barcode}(I_j)$ depend on the particular arrangement of points in the persistence diagram $\barcode$, so the number of realizations is not purely a function of the number of critical points, but also requires specifying the number of nested, paired critical points.

This nesting of paired critical points is captured via the \define{Elder Rule}, which provides a way of extracting persistent $H_0$ using persistent $\pi_0$.
Originally described using the \define{merge tree} associated to a function $f$, the Elder Rule also promises to deliver the decomposition of a persistent vector space (Definition~\ref{defn:pers-vect}) freely generated by a persistent set (Definition~\ref{defn:pers-set}) into indecomposables---a decomposition that exists in the finite-dimensional case by Crawley-Boevey's Theorem~\ref{thm:crawley-boevey}.
In the restricted setting of Morse sets (Definition~\ref{defn:morse-set}), we prove that this promise holds and give a new proof of the Elder Rule in Theorem~\ref{thm:elder-rule}.
In Remark~\ref{rmk:elder-sheaf}, we note that the Elder Rule also provides a method of finding the indecomposable summands of the pushforward of the constant sheaf along a map from a tree $\pi: T \to \R$.
In Remark~\ref{rmk:stratified-elder-map} we note how the above formula suggests a natural stratification of the space of barcodes based on the containment poset associated to $\barcode$.

In pursuit of our main theorem, Theorem~\ref{thm:count-persistence}, we do other things as well: We count the number of merge trees that have the same barcode in Theorem~\ref{thm:counting-merge-trees}, and introduce the notion of a \define{chiral merge tree}, which is a merge tree with a handedness decorating each of its edges.

\subsection{Acknowledgements}

The first real step forward on this problem happened in conversations with Hans Riess, who was an undergraduate at Duke at the time.
It was there that the author developed a method for constructing at least one function realizing any suitable barcode.
A full characterization of the fiber of the persistence map was obtained after conversations with Yuliy Baryshnikov at ICERM and with John Harer at Duke.
Yuliy Baryshnikov first introduced the concept of a chiral barcode associated to a time series in June 2015 in a talk at a conference on Geometry and Data Analysis at the Stevanovich Center for Financial Mathematics, 
the author then used the underlying concept of a chiral merge tree to create multiple realizations of a single barcode.
Finally, John Harer helped cinch the counting result by pointing out the restriction that the Elder Rule imposed on chiral merge trees.
The author believes that the conversations with Yuliy and John provided critical insights at the heart of this paper.

Other discussions have benefited subsequent revisions of this paper.
Rachel Levanger pointed out a deficiency in an earlier version of Definition~\ref{defn:merge-tree}.
Elizabeth Munch and Anastasios Stefanou first made the connection between interleavings and certain labeled merge trees that eventually appeared in their work~\cite{Munch2018}.
The author received several helpful comments from each of the anonymous referees, which helped improve the paper. 
One comment in particular caused the author to think more carefully about maps of ordered merge trees.

\section{The Persistence Map}

Let $(\R,\leq)$ denote the partially ordered set of real numbers, where $r\leq t$ if and only if $t-r$ is a non-negative real number.
We also view $\R$ as a category with one object for every real number and with a unique morphism $r \to t$ whenever $r\leq t$.

\begin{defn}
Let $\Ccat$ be any category.
A \define{persistence module valued in $\Ccat$} is a functor 
$$F: (\R,\leq) \to \Ccat.$$
In more detail, a persistence module valued in $\Ccat$ specifies an object $F(t)$ of $\Ccat$ for every time $t\in \R$ and a \emph{shift} morphism $\varphi_{tr}: F(r) \to F(t)$ for every relation $r\leq t$; any morphism can be a shift morphism so the only thing that the term ``shift'' signifies is how to shift the object at $r$ down the real line to the object at $t$.
The collection of shift morphisms defining $F$ must satisfy $\varphi_{tt}=\id_{F(t)}$ and $\varphi_{tr} =\varphi_{ts} \circ \varphi_{sr}$ for all $r\leq s \leq t$.

A map $\psi:F \Rightarrow F'$ of persistence modules is simply a natural transformation of functors, i.e.~for every $t\in \R$ a morphism $\psi_t:F(t) \to F'(t)$ is given and these morphisms are compatible with the shift morphisms in the sense that the following square commutes:
\[
\xymatrix{F(r) \ar[r]^{\varphi_{tr}} \ar[d]_{\psi_r} & F(t) \ar[d]^{\psi_t} \\ 
F'(r) \ar[r]_{\varphi'_{tr}} & F'(t)}
\]
In other words, for every pair $r\leq t$ we have $\varphi'_{tr}\circ \psi_r = \psi_t \circ \varphi_{tr}$.
We denote the category of persistence modules valued in $\Ccat$ by $\Pers(\Ccat)$.
\end{defn}

We now note some important special cases of this definition and the corresponding special language.

\subsection{Persistent Sets}

\begin{defn}[Persistent Sets]\label{defn:pers-set}
Let $\Set$ denote the category of sets and set maps.
A \define{persistent set} is a functor 
$$S: (\R,\leq) \to \Set.$$
A map of persistent sets is simply a natural transformation.
\end{defn}

\begin{rmk}[Maps to and from the Empty Set]\label{rmk:no-maps-empty}
Recall that the definition of a map of sets $f:X \to Y$ is that for every element $x$ of $X$ a unique element $f(x)$ in $Y$ is associated.
When $X$ is the empty set, this condition is automatically satisfied because there are no elements. In this case we conclude that for each set $Y$ there is a unique map $\varnothing \to Y$ called the \define{empty map}.
On the other hand the only map \emph{to} the empty set is \emph{from} the empty set. 
There can be no maps to the empty set from a non-empty set.
The implication of this observation is that for a persistent set $S$ if there is a $t\in \R$ so that $S(t)\neq \varnothing$, then $S(r)\neq \emptyset$ for every $r\geq t$.
\end{rmk}

\begin{ex}[Persistent Path Components]\label{ex:ppc}
Suppose $X$ is a topological space and $f:X \to \R$ is a function.
We can study the persistent set $S: (\R,\leq) \to \Set$ that assigns to every value $t$ the set of path components $S(t):= \pi_0(f^{-1}(-\infty,t])$ and to every pair of numbers $r\leq t$ the map $\varphi_{tr}: \pi_0(f^{-1}(-\infty,r]) \to \pi_0(f^{-1}(-\infty,t])$ on path components.
\end{ex}

\subsection{Persistent Vector Spaces}

When every set $S(t)$ has the structure of a vector space and when every map is a linear transformation of vector spaces, then we have the notion of a persistent vector space, which is called a \define{persistence module} in the literature.

\begin{defn}[Persistent Vector Spaces]\label{defn:pers-vect}
Let $\Vect_{\field}$ denote the category of $\field$-vector spaces and $\field$-linear transformations.
A \define{persistent vector space} is a functor 
$$F:(\R,\leq) \to \Vect_{\field}.$$
A map of persistent vector spaces is a natural transformation.
We often omit mention of the field $\field$ since it is arbitrary for this paper.
\end{defn}

\begin{rmk}[Terminology]
To the author's knowledge, the terms ``persistent set'' and ``persistent vector space'' are being used here for the first time.
The reasons for bringing this new terminology into circulation are twofold: Firstly, the pronounciation of ``persistent set'' sounds better than ``persistence set,'' which would be the logical shortening of ``persistence module valued in $\Set$.''
Secondly, an attempt to maintain grammatical consistency suggests that we use ``persistent vector spaces'' following ``persistent homology,'' rather than ``persistence modules.''
\end{rmk}

\begin{defn}\label{defn:free}
Let $S$ be a persistent set.
We get the persistent vector space \define{freely generated by $S$}, written $F_S$, by letting $F_S(t)$ be the vector space freely generated by $S(t)$ and letting the shift maps of $F_S$ be the linear maps freely generated by shift maps of $S$.
\end{defn}

\begin{defn}\label{ex:interval-modules}
Let $I\subseteq \R$ be an interval, i.e.~if $r,t\in I$ and $r\leq s \leq t$, then $s\in I$.
Associated to any interval in $\R$ is a persistent vector space that assigns $\field$ to every $t\in I$ and assigns the zero vector space to $t\notin I$.
For pairs of numbers $r\leq t$, both of which are in $I$, then the associated shift map is the identity map, otherwise it is the zero map.
In the literature this is called the \define{interval module} $I_{\field}$.
\end{defn}

\begin{rmk}
Note that if the interval $I\subset \R$ has $\sup I < \infty$, then the interval module $I_{\Bbbk}$ cannot be freely generated by any persistent set.
This is essentially due to Remark~\ref{rmk:no-maps-empty} because if $I_{\Bbbk}$ is freely generated by a persistent set $S$, then for $t\in I$ we have that $S(t)=\{\star\}$ the set with one element and no later set can be non-empty.

This also illustrates an important difference between the category of vector spaces and the category of sets. In $\Vect$ the zero vector space is both initial and terminal---meaning there are unique maps from and to the zero vector space---but in $\Set$ the empty set is initial and the one point set is terminal.
\end{rmk}

\subsection{Persistent Homology}

Although this narrative is anachronistic, persistent homology can be viewed as resting on the following technical result of Crawley-Boevey, which provides a finite characterization of any finite-dimensional persistent vector space.

\begin{thm}[{\cite[Thm.~1.1.]{crawley2012decomposition}}]\label{thm:crawley-boevey}
Any pointwise finite-dimensional persistent vector space (i.e.~persistence module) has a uniquely associated direct sum decomposition into interval modules.
\end{thm}

\begin{rmk}
Interval modules are indecomposable representations of the poset $(\R,\leq)$.
In the setting where our persistent vector spaces are pullbacks of representations of $A_n$-type quivers, the above result follows from earlier work of Gabriel~\cite{gabriel1972}.
\end{rmk}

\begin{defn}\label{defn:barcodes}
The intervals that appear in the direct sum decomposition of a persistent vector space $F$ guaranteed by Theorem~\ref{thm:crawley-boevey} define the \define{barcode associated to $F$}, written $\barcode(F)$.
In general, a \define{barcode} is any multi-set of intervals $\barcode=\{(I_j;n_j)\}$; here $n_j$ is the (finite) number of repetitions of the interval $I_j$. 
We will often work in the generic setting where $n_j=1$ or $0$, and drop this extra information.
We denote the set of all barcodes by $\BCspace$.
\end{defn}

\begin{figure}[h]
\centering
\includegraphics[width=.5\textwidth]{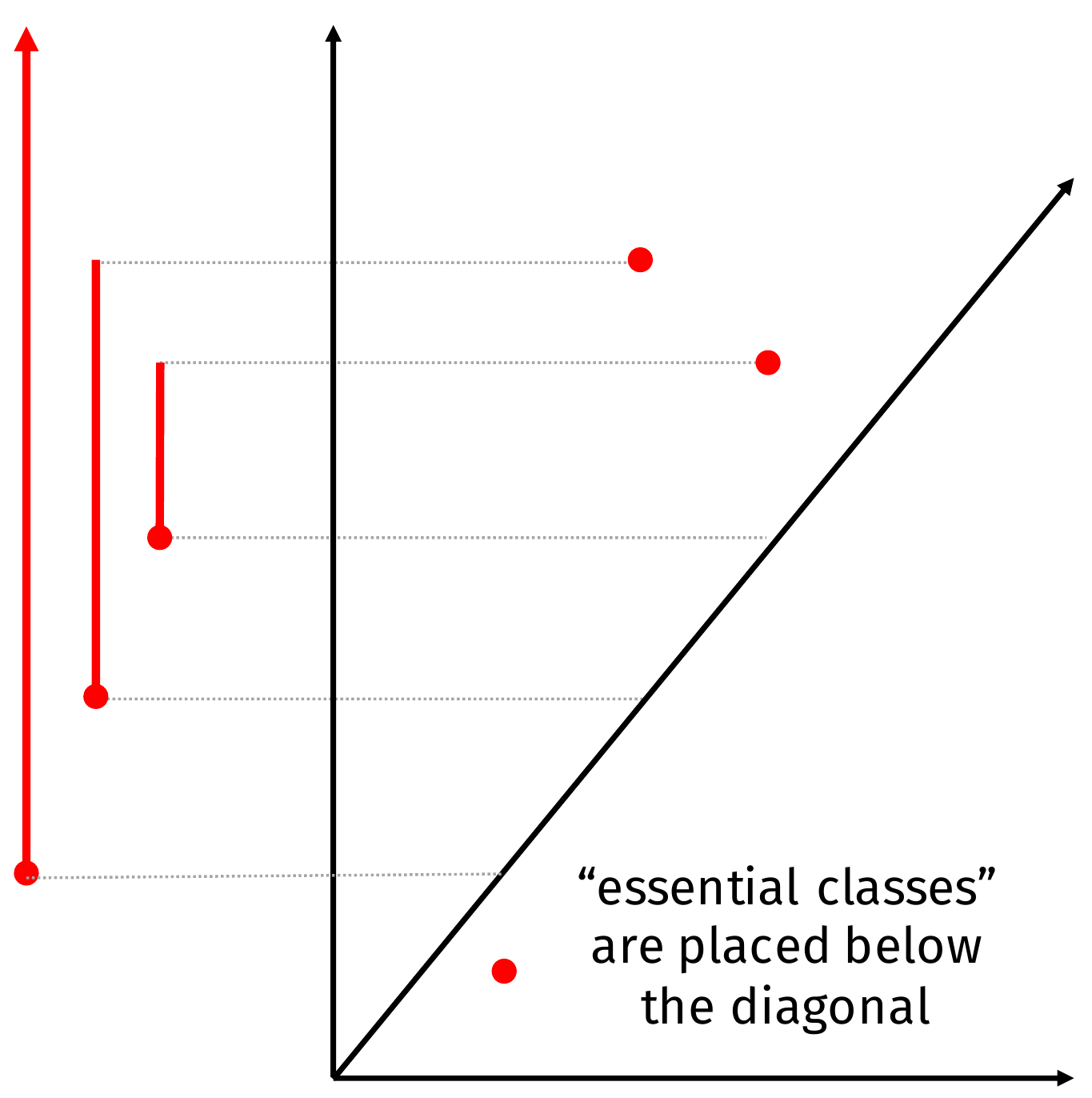}
\caption{The Persistence Diagram Associated to a Barcode.}
\label{fig:barcode-to-PD}
\end{figure}

\begin{rmk}[Persistence Diagrams]
Given a barcode $\barcode=\{I_j\}$ we can associate a collection of points in the extended plane $\mathbb{E}^2:=\R\cup \{-\infty\} \times \R \cup \{\infty\}$.
This is done by taking each interval $I_j\in\barcode$ to the coordinate-pair $(b_j,d_j)\in \mathbb{E}^2$ where $b_j=\inf I_j$ and $d_j=\sup I_j$.
This is called the \define{persistence diagram}~\cite{cohen2007stability} and it encodes the rank function of a persistent vector space as follows: If $\barcode(F)$ is the barcode associated to a persistent vector space, then the rank of the map $\varphi_{tr}: F(r) \to F(t)$ is given by the number of points in the persistence diagram up and to the left of the point $(r,t)\in \mathbb{E}^2$.
It should be noted that persistence was initially defined in terms of images of the shift maps, see~\cite{edelsbrunner2000topological}, and the language of barcodes appeared somewhat later~\cite{zomorodian2005computing}.
We pass back and forth between these two perspectives freely, using whichever representation works best in the moment, but we remark that each has their advantages. 
One deficit of persistence diagrams is that intervals with different endpoint types---say $[b,d)$ and $[b,d]$---are associated to the same point $(b,d)$ in the persistence diagram.
This difference doesn't matter so much in practice because most intervals that arise by filtering a space by a Morse function are closed on the left and open on the right, but in level-set persistence the endpoint type encodes topology, see~\cite{TDAcosheaves} for more on that.
However Amit Patel has shown in~\cite{generalized-PD} that persistence diagrams can be defined more generalized to settings where the barcode doesn't make sense.
\end{rmk}

\begin{defn}\label{defn:ph}
Associated to a function $f:X\to \R$ is the \define{persistent homology in degree $i$}, which is the persistent vector space
\[
F_i : (\R,\leq) \to \Vect \qquad t \rightsquigarrow H_i(f^{-1}(-\infty,t]).
\]
Here $H_i(-)$ denotes the homology functor in degree $i$ taken with coefficients in a field $\field$.
More accurately, the \define{persistent homology from $r\leq t$} is the \emph{image} of the map
\[
\varphi_{tr} : H_i(f^{-1}(-\infty,r]) \to H_i(f^{-1}(-\infty,t]).
\]
\end{defn}

Since $H_0$ is the vector space freely generated by $\pi_0$, we can think of $F_0$ as the persistent vector space freely generated by the persistent set described in Example~\ref{ex:ppc}.

\begin{defn}\label{defn:pers-map}
Suppose $X$ is a compact simplicial complex and let $L^{\infty}(X)$ denote the space of continuous functions on $X$, equipped with the sup-norm.
The \define{persistence map in degree $i$} is the map that takes each function $f$ to its persistent homology in degree $i$ and then its uniquely associated barcode:
\[
PH_i : \qquad  f\in L^{\infty}(X) \qquad \rightsquigarrow \qquad F_i\in \Pers(\vect) \qquad \rightsquigarrow \qquad \barcode(F_i) \in \BCspace
\]
\end{defn}

\begin{rmk}
This map is actually $1$-Lipschitz, after one equips the category of persistent vector spaces with the \define{interleaving distance}, see~\cite{chazal2012structure} for an overview.
Moreover, the last map, which takes a persistent vector space to its barcode, is an isometry once the space of barcodes is equipped with the \define{bottleneck distance}~\cite{Lesnick2015}.
However, we will not make explicit use of these metrics, and only refer to them in passing, e.g.~Remarks~\ref{rmk:lipschitz-elder-map},
\ref{rmk:stratified-elder-map} and \ref{rmk:metrics-for-OMTs}.
\end{rmk}


\section{The Elder Rule for Morse Sets}

In this section we give a treatment of the \define{Elder Rule} that differs in flavor from the description given in~\cite[p.~150]{edelsbrunner2010computational}.
In our language, the Elder Rule allows us to associate a barcode directly to a persistent set $S$, without first considering the persistent vector space $F_S$ that $S$ freely generates and by applying Theorem~\ref{thm:crawley-boevey}.
The fact that these two ways of associating a barcode to a persistent set agree is the content of Theorem~\ref{thm:elder-rule}.

Our treatment makes use of standard constructions in the theory of partially-ordered sets.
To that end, we recall some standard results and terminology associated to posets.
An \define{up-set} $U\subset (P,\preceq)$ is any set where $x\in U$ and $x\leq y$ jointly imply that $y\in U$.
A \define{principal up-set} is any set of the form $U_x:=\{y\in P \mid x\leq y\}$.
A \define{chain} $C$ in $(P,\preceq)$ is a subset of $P$ that is totally-ordered upon restriction of the partial order to $C$.
A chain $C$ is \define{maximal} if whenever we have two chains $C\subseteq C'$, then $C=C'$.

\begin{defn}\label{defn:P_S}
Let $S$ be a persistent set.
Consider the set $P_S=\bigcup_{t\in \R} S(t)\times \{t\}$.
We define a partial order $\preceq$ on $P_S$ by declaring $(x,r)\preceq (y,t)$ whenever $t-r$ is non-negative and $\varphi_{tr}(x)=y$.
\end{defn}

\begin{rmk}
Any functor $F:(P,\preceq) \to \Set$ can be thought of as a sheaf of sets in the Alexandrov topology.
This is the topology whose basis is given by principal up-sets $U_p$.
The construction above is a special case of the observation that the \'etal\'e space associated to a sheaf over a poset is also a poset.
Indeed the \'etal\'e space of $F:(P,\preceq) \to \Set$ is simply the disjoint union
\[
	E(F):= \bigsqcup_{p\in P} F(p) = \bigcup_{p\in P} F(p) \times \{p\}
\]
We can then define a partial order $\preceq'$ on $E(F)$ that extends the partial order on $(P,\preceq)$ by declaring $(x,p) \preceq' (y,q)$ whenever $p\preceq q$ and $\varphi_{qp}(x)=y$.
\end{rmk}

Recall that a map of posets is one that is order-preserving. Equivalently, a map of posets is a continuous map in the Alexandrov topology.

\begin{defn}\label{defn:support}
Let $\pi: (P,\preceq) \to (\R,\leq)$ be a map of posets.
Whenever we write $P(t)$ we mean $\pi^{-1}(t)$.
If $P(t)\neq \varnothing$, then we say $P$ is \define{supported} at $t$.
The \define{support} of $P$ is defined to be
\[
\supp(P):=\{t\in \R \mid \pi^{-1}(t)\neq \varnothing \}.
\]
Note that $\supp(P)$ need not be closed in $\R$ with the Euclidean or the Alexandrov topology.
\end{defn}

We now define what it means for a chain to be older than another chain.

\begin{defn}\label{defn:older}
Let $\pi: (P,\preceq) \to (\R,\leq)$ be a map of posets and
let $C_1$ and $C_2$ be two chains in $P$.
We say that $C_1$ is \define{older} than $C_2$ if there exists a $t_1\in \supp(C_1)$ such that $t_1 < t_2$ for all $t_2\in \supp(C_2)$.
\end{defn}

We view the Elder Rule as a method for partitioning the poset $P_S$ associated to a persistent set $S$ into chains.
In order to describe this method inductively, and without special case analysis, we introduce the notion of a Morse set, which is a special case of a constructible persistence module valued in $\Set$, see~\cite[Def.~2.2]{generalized-PD}.
Intuitively speaking, a Morse set is an abstraction of the persistent set of path components associated to a Morse function on a compact, connected manifold; cf.~Example~\ref{ex:ppc}.

Recall that a \define{constant} functor is a functor that assigns to every object in its domain of definition a single object and sends every morphism to the identity map on that object.

\begin{defn}\label{defn:morse-set}
Let $\set$ denote the category of finite sets and set maps.
A \define{Morse set} is a functor $S:(\R,\leq) \to \set$ along with a finite sequence of times $\tau =\{\tau_1 < \cdots < \tau_n\}$ so that
\begin{enumerate}
	\item $S(t)=\varnothing$ for $t<\tau_1$,
	\item $S|_{[\tau_i,\tau_{i+1})}$ is naturally isomorphic to the constant functor with value $S(\tau_i)$ for $i=1,\ldots, n-1$,
	\item $S(t)=\{\ast\}$ for $t\geq \tau_n$,
	\item for all $x\in S(\tau_{i+1})$ the fiber $\varphi_{i+1,i}^{-1}(x)$ has cardinality 0,1, or 2, and
	\item each element $y\in S(t)$ is contained in a unique, oldest maximal chain $C_y$ in the associated poset $P_S$ defined in Definition~\ref{defn:P_S}.
\end{enumerate}
Moreover, we assume that the set of times $\tau \subset \R$ is the minimal set of times making the above statements true.
We let $\mathcal{S}$ denote the set of isomorphism classes of Morse sets.
\end{defn}

\begin{rmk}[Interpretation of the Hypotheses]
As already mentioned, the Morse set is supposed to represent the properties of the persistent set of path components associated to a Morse function---where crtical points are further assumed to have distinct critical values.
In particular, when we desribe the merge tree $T$ associated to a Morse Set in Lemma~\ref{lem:MS-to-MT}, the fifth hypothesis expresses the requirement that each leaf node $v$ has a uniquely associated value $\pi(v)$.
\end{rmk}

\begin{defn}[The Elder Rule]\label{defn:elder-rule}
Let $S$ be a Morse set, constructible with respect to the times $\tau =\{\tau_1 < \cdots < \tau_n\}$.
The \define{Elder Rule} gives the following inductive chain decomposition of the poset $P_S$ described in Definition~\ref{defn:P_S}.
For any poset map $\pi:(P,\preceq) \to (\R,\leq)$, we define $P(t):=\pi^{-1}(t)$.
\begin{itemize}
	\item Let $P_1$ be the poset $P_S$ and let $C_1$ be the unique, oldest maximal chain containing the element $\star \in P_1(\tau_n)$.
	\item Let $P_{i+1}=P_i-C_i$, i.e.~the poset $P_i$ with the chain $C_i$ removed.
	\item By the fourth and fifth hypotheses of Definition~\ref{defn:morse-set}, $P_{i+1}$ has a unique element $\star_{n-i} \in P_{i+1}(\tau_{n-i})$ that is contained in a unique, oldest chain $C_{i+1}$.
\end{itemize}
Let $\Ecode =\{\pi(C_i)\}$ to be the set of intervals associated to the chains $C_i$ via projection along $\pi$.
This defines the barcode associated to $S$ by the Elder Rule.
\end{defn}

\begin{figure}[h]
\centering
\includegraphics[width=.8\textwidth]{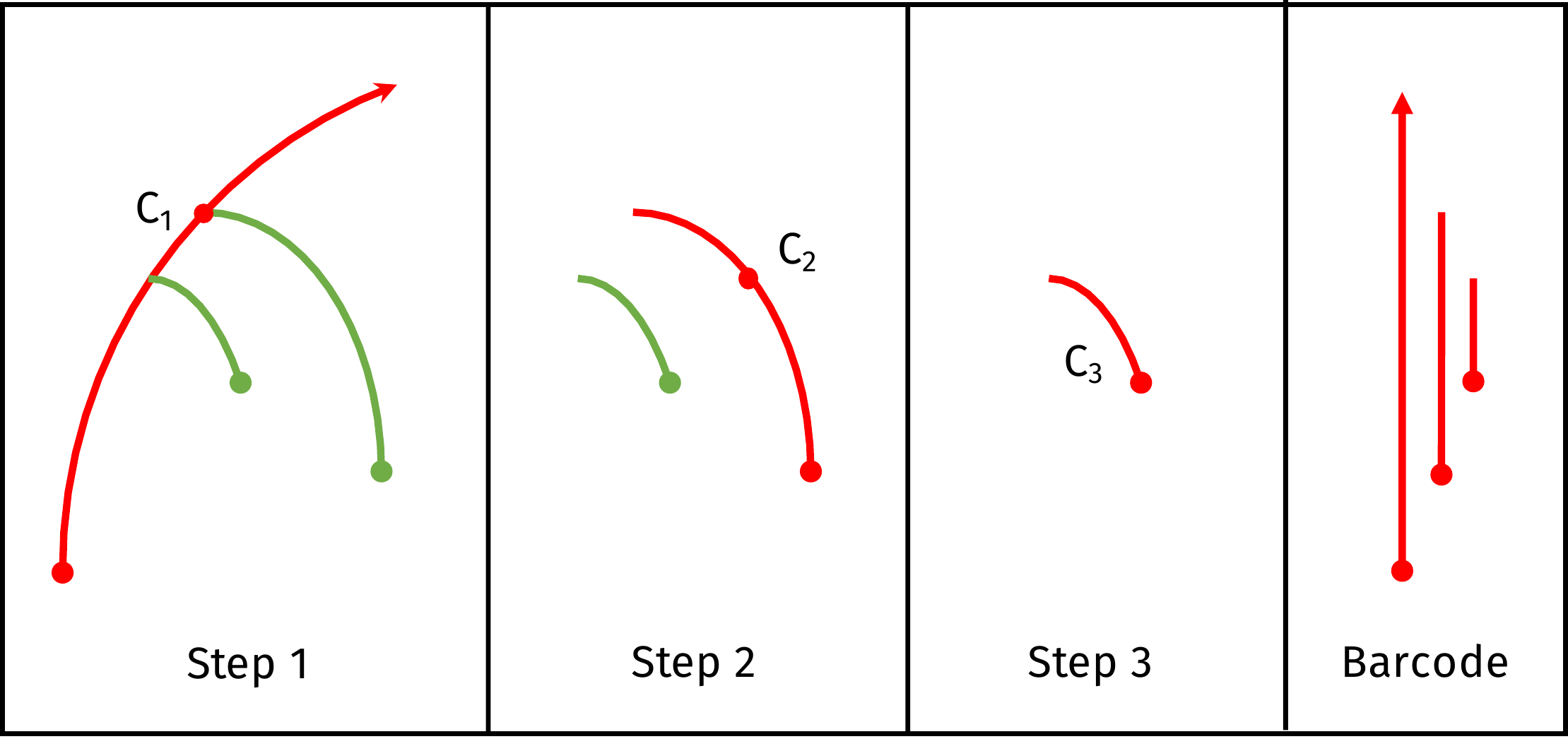}
\caption{Inductive chain decomposition of a Morse set and its projected barcode.}
\label{fig:elder-rule-ex}
\end{figure}

\begin{rmk}
The hypotheses in the definition of a Morse set are adapted to make the statement and proof of the Elder rule as simple as possible, without having to handle extra unnecessary cases or make non-canonical choices.
Relaxation of the fifth hypothesis of a Morse set, for example, could allow the Elder rule to be applied to persistent sets whose merge trees have several leaf nodes with the same function value and this would introduce non-canonical choices of picking one of possible several equally old chains. This is common in practice and it would not be too difficult to change the statement and proof of the Elder rule to allow this, but it would complicate the presentation considerably, so we have contented ourselves with the generic case.
\end{rmk}

\begin{figure}[h]
\centering
\includegraphics[width=\textwidth]{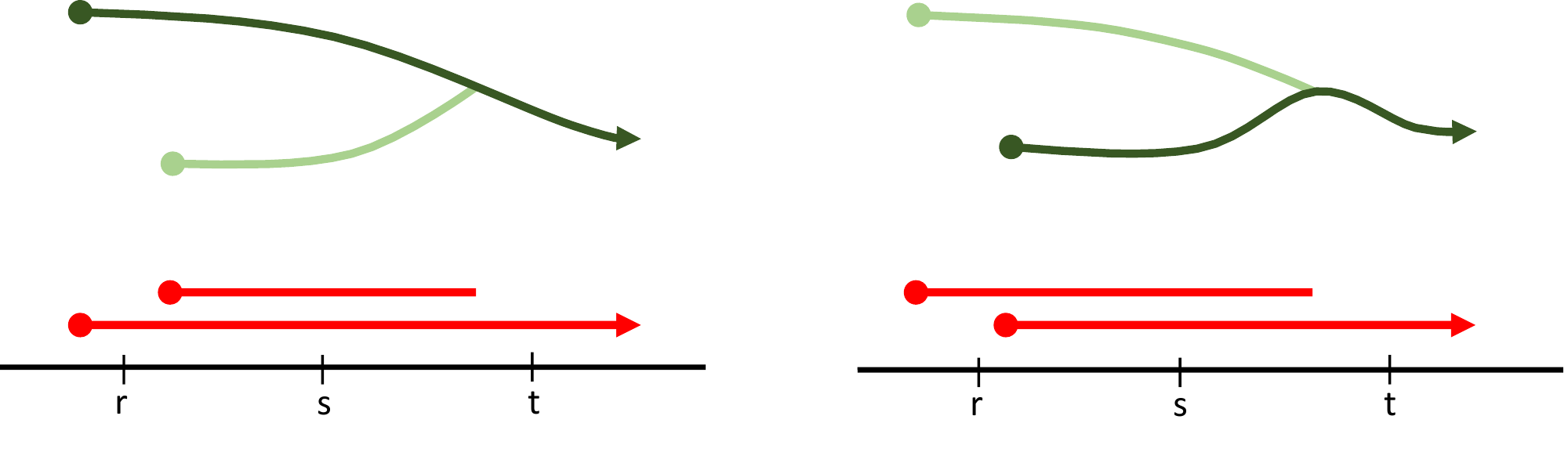}
\caption{The Elder Rule is used to compute the barcode on the left. This accurately determines the rank of the map $\varphi_{tr}$. The barcode on the right does not.}
\label{fig:elder-chains}
\end{figure}

\begin{ex}
In Figure~\ref{fig:elder-chains} we have the same persistent set $S$, but with two different partitions into chains. 
By considering the projection of the chain decomposition, we get two different barcodes $\Ecode$ and $\barcode$, depicted to the left and right, respectively, in Figure~\ref{fig:elder-chains}.

If one considers the persistent vector spaces $V_{\Ecode}$ and $V_{\barcode}$ generated using the intervals in their respective barcodes, one can see an important distinction between the two: Only $V_{\Ecode}$ has the property that it is isomorphic to $F_S$---the persistent vector space freely generated by the persistent set $S$.
To see that $V_{\barcode}\ncong F_S$ note that the linear map from $r$ to $t$ is zero for $V_{\barcode}$ and non-zero for $F_S$.
\end{ex}

We now prove that the Elder Rule is true. 
To the author's knowledge there is no precise proof of the Elder Rule in the literature. 
Although the correctness of the Elder Rule is guaranteed by the design of the original persistence algorithm, we provide a proof of a different flavor.
We note that the proof techniques used below are quite general and should work for persistent sets more general than Morse sets, as soon as the Elder rule is properly modified to allow more general persistent sets, but we leave this generalization to future work.

\begin{thm}[The Elder Rule]\label{thm:elder-rule}
Let $S$ be a Morse set. 
Let $F_S$ be the persistent vector space freely generated by $S$.
If $V_{\Ecode}$ is the persistent vector space associated to the barcode $\Ecode$ generated by the Elder Rule, then 
\[
F_S \cong V_{\Ecode}.
\]
\end{thm}
\begin{proof}
Note that for each $t\in\R$, a basis for $V_{\Ecode}(t)$ is given by the chains $C_i$ supported at $t$.
This basis is also well-adapted to the shift maps internal to $V_{\Ecode}$, because $\eta_{tr}:V_{\Ecode}(r) \to V_{\Ecode}(t)$ always takes $C_i(r)$ to $C_i(t)$, with the understanding that if $C_i(t)=\varnothing$ in $P_S(t)$, then $C_i(t)=0$ in $V_{\Ecode}(t)$.
With this basis the shift map $\eta_{tr}$ is diagonal for every $r\leq t$.

We now turn our attention to the persistent vector space $F_S$. The shift maps $\varphi_{tr}$ do not take chains to chains; rather, they take certain up-sets associated to chains to up-sets with a re-indexing rule.
To introduce some more notation, let $U_i=\{y \in P_S \mid \exists x\in C_i s.t.~\varphi_{tr}(x)=y\}$; this is the up-set associated to the chain $C_i$.
Note that $\{U_i(t)\}$ spans $F_S(t)$ for all $t$, again with the understanding that if $U_i(t)=\varnothing$ in $P_S$, then $U_i(t)=0$ in $F_S(t)$.
The only way in which the set $\{U_i(t)\}$ fails to be a basis is when there exist chains $C_i$ and $C_j$ with associated upsets $U_i$ and $U_j$ that intersect at time $t$ and then the same vector appears twice in the set $\{U_i(t)\}$.
We correct this by adopting the convention that whenever $U_i(t)=U_j(t)$ we choose the lower index $\min\{i,j\}$.
With this choice of basis for all $t$, one can see that the linearized shift maps $\varphi_{tr}$ take each $U_i(r)$ to $U_k(t)$ where $k=\min\{j \mid U_j(t)=U_i(t)\}$.

Using the convention that $C_0=\varnothing$ and hence $U_0=\varnothing$, the change of basis map taking $V_{\Ecode}(t)$ to $F_S(t)$ is defined on chains supported at $t$ by
\[
\beta_t: C_{i}(t) \mapsto U_{i}(t)-U_{i-1}(t) \qquad \text{for} \qquad i=1,\ldots,N,
\]
where $N$ is the number of chains in the decomposition of $P_S$ given by the Elder Rule in Definition~\ref{defn:elder-rule}.
We now need to show that for all $r\leq t$ and $i\in\{1,\ldots, N\}$ that
\begin{equation}\label{eqn:elder-commutes}
\varphi_{tr} \circ \beta_r (C_i(r)) = \beta_t \circ \eta_{tr} (C_i(r)).
\end{equation} 
If $C_i(r)=0$, then there is nothing to check. 
If $C_i(r) \neq 0$, then we distinguish two cases based on $t$. 
If $U_i$ and $U_{i-1}$ intersect at $t$, then $C_i(t)=0$. 
To see this, note that if $U_i(t)=U_{i-1}(t)$, then there is an older chain passing through $U_i(t)$, and hence $C_i$ is not supported at $t$.
In this case, the left hand side of Equation~\ref{eqn:elder-commutes} reads
\[ 
\varphi_{tr} \circ \beta_r (C_i(r)) = \varphi_{tr}(U_i(r) - U_{i-1}(r)) = U_i(t) - U_{i-1}(t) = 0.
\]
This in turn agrees with the right hand side of Equation~\ref{eqn:elder-commutes}:
\[
\beta_t \circ \eta_{tr} (C_i(r)) = \beta_t (C_i(t)) = \beta_t(0) = 0.
\]
The other case, where $U_i(t)\neq U_{i-1}(t)$, implies that $\varphi_{tr}(U_i(r)-U_{i-1}(r))=U_i(t)-U_{i-1}(t)$, which is exactly what $\beta_t(C_i(t))$ is defined to be.
This completes the proof.
\end{proof}


\section{Merge Trees and the Elder Map}

The poset $P_S$ in Definition~\ref{defn:P_S} serves as a total space for any persistent set $S$.
When $S$ is constructible, e.g.~Morse, we can associate a different total space to $S$ called the \define{merge tree} associated to $S$.
Below we define what a merge tree is in the abstract, without reference to persistent path components of a function $f:X \to \R$.

\begin{defn}\label{defn:merge-tree}
A \textbf{merge tree} consists of the following data:
\begin{itemize}
	\item A connected, locally-finite, contractible, one-dimensional cell complex $T$ with a distinguished edge $e_{\infty}$ without compact closure. In other words, if we remove $e_{\infty}$ and consider what remains $T \setminus e_{\infty}$, then this is a compact tree rooted at $v_{\infty}$, which is the one vertex incident to $e_{\infty}$.
	\item A continuous map $\pi: T \to \R$, where $\R$ has the Euclidean topology, that is orientation-preserving in the following sense: 
	\begin{itemize}
		\item $\pi(e_{\infty})=(t_{\infty},\infty)$ where $\pi(v_{\infty})=t_{\infty}$. 
		\item If $\gamma:[0,1] \to T$ is an injective map that starts at a vertex and ends at $v_{\infty}$, then $\pi\circ \gamma:[0,1] \to \R$ is monotonically increasing.
	\end{itemize}
\end{itemize}
Note that the second bullet point implies that every edge $e\subseteq T$ has a \define{length}, given by the diameter of its projection $\pi(e)\subset \R$. 
Also, we can define the \define{child} of a vertex $v\in T$ as any vertex $w$ connected by an edge to $v$ with $\pi(w) < \pi(v)$.
A map from a merge tree $\pi_1: T_1 \to \R$ to a merge tree $\pi_2 :T_2 \to \R$ is a map of spaces $\phi: T_1 \to T_2$ satisfying $\pi_1=\pi_2\circ \phi$, i.e.~this is a morphism in the category $\Top \downarrow \R$.
In the event that the map of spaces $\phi$ is a homeomorphism, we say the merge trees are isomorphic.
Note that this definition implies that if $T$ has an extra vertex that simply subdivides an edge, then it is isomorphic to a merge tree where that vertex is removed.
\end{defn}

\begin{lem}\label{lem:MS-to-MT}
Every Morse set $S$ has an associated merge tree $\pi: T \to \R$.
\end{lem}
\begin{proof}
We first form the disjoint union
$$S(t_1) \times [\tau_1,\tau_2] \sqcup \cdots \sqcup S(\tau_n) \times [t_n,\infty)$$
and then we impose the equivalence relation that $(x,\tau_{i+1}) \in S(\tau_i)\times \{\tau_{i+1}\}$ is identified with $(y,\tau_{i+1})\in S(\tau_{i+1})\times \{\tau_{i+1}\}$ if and only if $\varphi_{i+1,i}(x)=y$.
This defines the total space $T$.
Projection onto the second factor defines the map $\pi: T \to \R$.
\end{proof}

It is clear that the hypotheses of a Morse set make $T$ into a binary tree.
We set this aside as a special definition.

\begin{defn}\label{defn:Morse-tree}
If $\pi: T\to \R$ is the merge tree associated to a Morse set, then we call it a \define{Morse tree}.
Let $\mathcal{T}$ denote the set of isomorphism classes of Morse trees.
\end{defn}

\begin{lem}\label{lem:morse-set-tree-equiv}
The set of Morse sets $\mathcal{S}$ and the set of Morse trees $\mathcal{T}$ are isomorphic.
\end{lem}
\begin{proof}
Note that Morse sets constitute a subcategory of $\Pers(\set)$ and Morse trees constitute a subcategory of the over category $\Top \downarrow \R$.
The lemma follows from the more general statement that these subcategories are equivalent, which implies they have isomorphic sets of isomorphism classes. We now described the main ideas of the equivalence and leave the remaining details to the reader. 

Consider the merge tree $\pi: T \to \R$ associated to a Morse set $S$. 
Naturally associated to $T$ is a persistent set $S'(t):=\pi_0(\pi^{-1}(-\infty,t])$.
Since every sub-level set $\pi^{-1}(-\infty,t]$ deformation retracts onto $\pi^{-1}(t)$, which is $S(t)$, we have that $S'(t)\cong S(t)$.
\end{proof}

Lemma~\ref{lem:morse-set-tree-equiv} implies that to every Morse tree we can associate a barcode, via the Elder Rule.

\begin{defn}\label{defn:tree-to-barcode}
Recall that $\BCspace$ denotes the set of all possible barcodes.
The Elder Rule defines the \define{Elder map} from the set of (isomorphism classes of) Morse Trees $\MTspace$ to the set of barcodes $\BCspace$:
$$\Xi:\MTspace \to \BCspace$$
\end{defn}

\begin{rmk}[The Elder Map is Lipschitz]\label{rmk:lipschitz-elder-map}
The Elder map is also $1$-Lipschitz, and hence continuous, by using the \define{interleaving distance} on merge trees~\cite{morozov2013interleaving}.
\end{rmk}

\begin{rmk}[A Sheaf-Theoretic Aside]\label{rmk:elder-sheaf}
Let $\pi: T \to \R$ be a Morse tree.
Let $\Bbbk_T$ be the constant sheaf on $T$.
Consider the pushforward $\pi_*\Bbbk_T$ along the map $\pi:T \to \R$, which is constructible with respect to the stratification of $\R$ given by the set of times $\tau$ used in the definition of a Morse set.
Using results of~\cite{curry-patel-CCC}, we can describe this sheaf in terms of a zig-zag module, which in this case is a sequence of vector spaces and maps of the form
	\[
	F_S(\tau_1) \leftarrow F_S(\tau_1) \rightarrow F(\tau_2) \leftarrow \cdots \rightarrow F_S(\tau_n) \leftarrow F_S(\tau_n)
	\]
We note that maps pointing to the left are always invertible, so we can regard this is as a persistent vector space.
Consequently, appealing to Theorem~\ref{thm:elder-rule}, the Elder Rule provides a decomposition of the pushforward of the constant sheaf along $\pi$ into indecomposable sheaves over $\R$.
\end{rmk}

\begin{figure}[h]
\centering
\includegraphics[width=.7\textwidth]{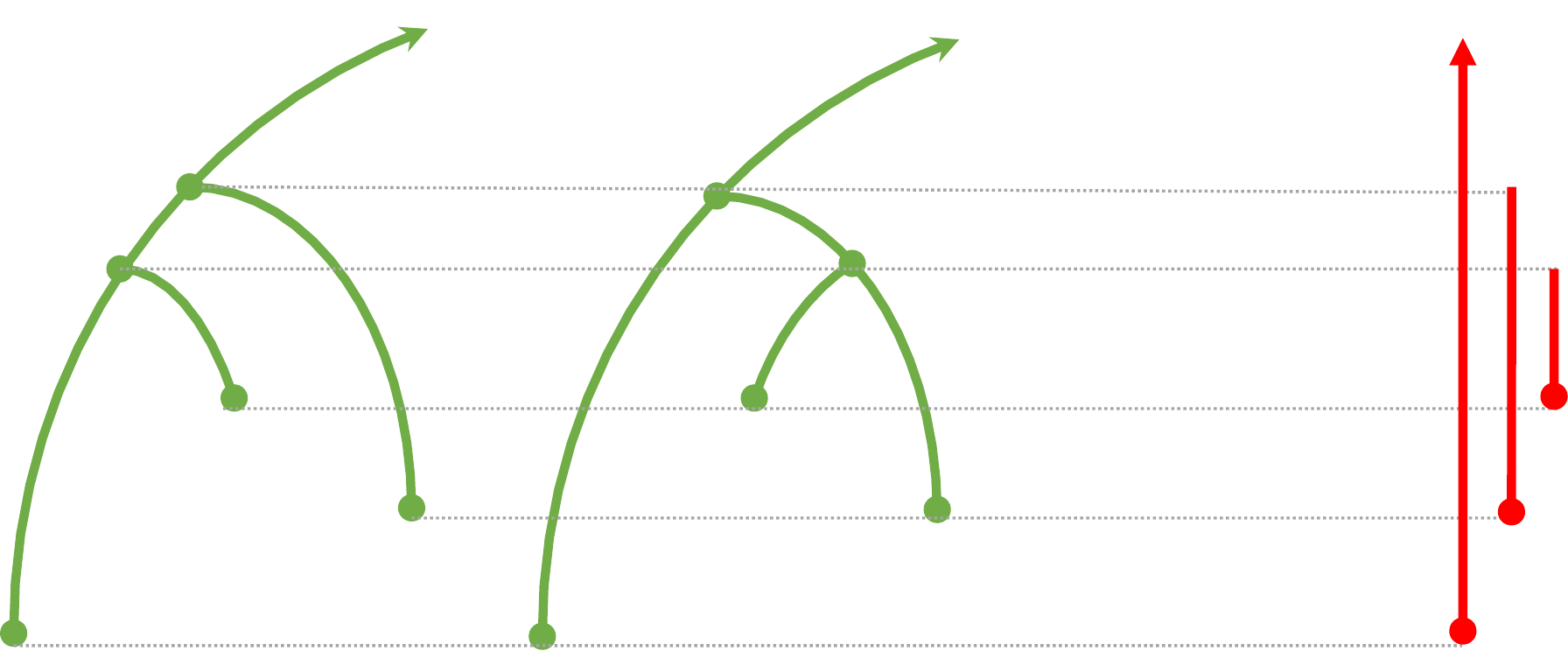}
\caption{Fiber of the Elder Map over a Barcode}
\label{fig:MT-fiber-example}
\end{figure}

We now describe the fiber of the Elder map $\Xi: \MTspace \to \BCspace$.

\begin{thm}\label{thm:counting-merge-trees}
Let $\barcode=\{I_j\}_{j=1}^N$ be a barcode where $I_1=[b_1,\infty)$ and where for $j\geq 2$ we have a strict containment $I_j=[b_j,d_j)\subset I_1$ with $d_j > d_{j+1}$.
We denote the set of intervals in $\barcode$ containing $I_{j}$ by $C_{\barcode}(I_j):=\{I_k\in\barcode \mid I_k \supset I_j \}$ and let $\mu_{\barcode}(I_j)$ denote the cardinality of this set.
The number of merge trees realizing the barcode $\barcode$ is
\[
|\Xi^{-1}(\barcode)| = \prod_{j=2}^N \mu_{\barcode}(I_j).
\]
\end{thm}

\begin{proof}
We construct a Morse tree $\pi:T \to \R$ inductively from $\barcode$ and enumerate all the possible choices along the way.
We note that the intervals in $\barcode$ are already ordered by their right-hand endpoints via their index $j$.
This is why we start by setting $T_1$ to be the interval $I_1=[b_1,\infty)$, equipped with the Euclidean topology.

To construct $T_{j+1}$ from $T_j$ we must select a place to attach the interval $I_{j+1}\in\barcode$ to $T_j$.
Note that every interval $I_k \in C_{\calB}(I_{j+1})$ has a right-hand endpoint that is to the right of $d_{j+1}$ and hence has already been used in the construction of $T_j$.
Once we select an $I_k \in C_{\calB}(I_{j+1})$---and we note that there were $\mu_{\barcode}(I_{j+1})$ choices---we can define
$$T_{j+1} := \left( T_j \sqcup \bar{I}_{j+1} \right)/\sim_k$$
where $\bar{I}_{j+1}=[b_{j+1},d_{j+1}]$.
The equivalence relation $\sim_k$ identifies the point $d_{j+1}\in I_k$ with the right-hand endpoint $d_{j+1} \in \bar{I}_{j+1}$.
Once $j=N$, we will have constructed a tree $T$ using the intervals in $\barcode$.
The map $\pi:T \to \R$ is simply the map that takes each $t\in I_k$ to $t\in \R$.

Note that every possible merge tree constructed in this way picks out a unique isomorphism class.
To see this, observe that if we have chosen $I_{k'}\neq I_k \in C_{\calB}(I_{j+1})$, then by hypothesis $d_{k'}\neq d_k$.
Attaching $I_{j+1}$ at $d_{j+1}\in I_{k'}$ produces an edge of length $d_{k'}-d_{j+1}$, which is an isomorphism invariant. This completes the proof of the theorem.
\end{proof}

\begin{rmk}
Note that any Morse function on a compact manifold, where we assume that Morse functions have distinct critical values for each critical point, will have a barcode in homological degree zero that satisfies the hypotheses of Theorem~\ref{thm:counting-merge-trees}.
Since Morse functions are dense in the space of all functions on a compact manifold, the assumptions on the barcode are generic.
\end{rmk}

\begin{ex}
In Figure~\ref{fig:MT-fiber-example} we have a barcode, depicted to the right.
Up to isomorphism, there are only two merge trees that realize this barcode, which agrees with the formula given in Theorem~\ref{thm:counting-merge-trees}.
\end{ex}

\begin{figure}[h]
\centering
\includegraphics[width=.9\textwidth]{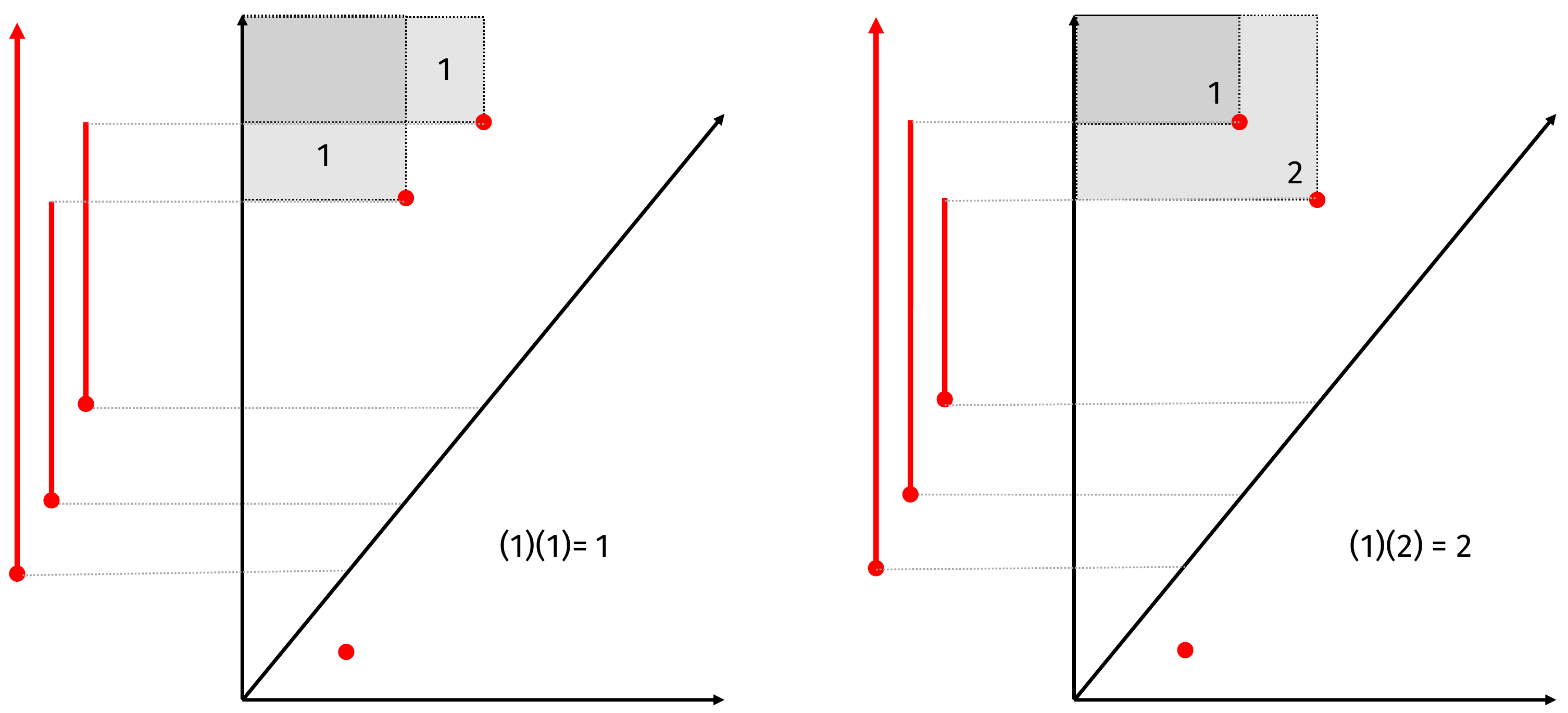}
\caption{Two barcodes that live in different strata of $\BCspace$.}
\label{fig:config-count}
\end{figure}

\begin{rmk}[Stratifying the Elder Map]\label{rmk:stratified-elder-map}
By using the persistence diagram perspective, we can think of barcodes in terms of configurations of points in $\mathbb{E}^2$. 
Note that the formula given in Theorem~\ref{thm:counting-merge-trees} defines a constructible function (see~\cite{schapira1991operations,euler} for more on constructible functions), which is constant on strata in a stratification of the space of barcodes $\BCspace$.
This stratification is defined using the number of points in the persistence diagram and the containment relations used to define $C_{\barcode}(I_j)$, which can be phrased in terms of linear inequalities.
For example, in Figure~\ref{fig:config-count} we see two barcodes, viewed as persistence diagrams, lying in different strata.
If we moved the upper most point in the right persistence diagram out of the shadow of the lower point, then we would move from one piece of the stratification of $\BCspace$ to another piece.
Strata are indexed by isomorphism classes of the \define{containment relation poset} associated to $\barcode$, which is the poset on $\barcode=\{I_j\}$ given by containment of intervals.
It appears that every partially ordered set of order dimension two with $N$ elements (and a unique maximal element) can be realized as the containment relation poset of a barcode $\barcode$ satisfying the hypotheses of Theorem~\ref{thm:counting-merge-trees} by embedding the Hasse diagram in the plane so that comparable elements are always up and to the left.
See~\cite{bayoumi1994counting} for an enumeration of order dimension two posets and hence (assuming the previous observation is true) strata of the space of barcodes.
\end{rmk}



\section{Ordered Persistent Sets and Chiral Merge Trees}

Suppose $f:[0,1] \to \R$ is a continuous map to $\R$, in the Euclidean topology.
For every real number $t\in \R$ we note that $f^{-1}(\infty,t]$ can be written as a disjoint union of closed intervals.
Because $[0,1]$ is totally ordered, we can order the intervals appearing in $f^{-1}(\infty,t]$ from left to right.
This implies that the persistent path components of $f$ are actually organized by a richer structure.

\begin{defn}
Let $\Ord$ denote the category of totally-ordered sets and order-preserving maps.
An \define{ordered persistent set} is any functor $W:(\R,\leq) \to \Ord$.
A map of ordered persistent sets is simply a natural transformation of functors.
\end{defn}

\begin{ex}
If $f:[0,1] \to \R$ is a continuous map, then set $W(t):=\pi_0(f^{-1}(-\infty, t])$ is totally ordered using the left-to-right ordering of the intervals making up the pre-image $f^{-1}(-\infty, t]\subseteq [0,1]$.
\end{ex}

This example begets the notion of an \define{ordered merge tree} and it's generic version, the \define{chiral merge tree}.
At the end of this section we describe what maps of these objects are.

\begin{defn}
An \define{ordered merge tree} is a merge tree $\pi: T \to \R$ with the additional data of specifying a total order $\{1,\ldots, n_v\}$ on the edges connecting $v$ to its children.
If a vertex $v\in T$ has two children $w_1$ and $w_2$, then an order amounts to an assignment of a left $L$ and right $R$ to the two edges connecting $v$ to its children, where we use the convention that $L<R$.
If every vertex in $T$ has two children, then we call an ordered merge tree a \define{chiral merge tree}.
Let $\CMTspace$ denote the set of all chiral merge trees.
\end{defn}

\begin{ex}
Suppose $f:[0,1] \to \R$ is a piece-wise linear (PL) function where every critical point has a distinct critical value.
Here a critical point $p\in [0,1]$ means that $p$ is either a local minimum or a local maximum, which is characterized by the existence of an open neighborhood $U\ni p$ for which either $f(p)\leq f(x)$ or $f(x) \leq f(p)$ for all $x\in U$.
The merge tree associated to $f$, possibly after removing unnecessary vertices, then has the structure of a chiral merge tree since at most two intervals merge when crossing a critical value associated to a local maximum.
\end{ex}

Defining the \define{Chiral Elder Map} via the composition
\[
\Xi_{\CMTspace} :\CMTspace \to \MTspace \to \BCspace
\]
we have the following corollary of Theorem~\ref{thm:counting-merge-trees}.

\begin{cor}\label{cor:counting-CMTs}
Let $\barcode$ be a barcode satisfying the hypotheses of Theorem~\ref{thm:counting-merge-trees}.
The number of chiral merge trees realizing $\barcode$ is
\[
|\Xi_{\CMTspace}^{-1}(\barcode)| = 2^{N-1}\prod_{j=2}^N\mu_{\barcode}(I_j).
\]
\end{cor}
\begin{proof}
The proof follows the proof of Theorem~\ref{thm:counting-merge-trees} with the exception that instead of there being $\mu_{\barcode}(I_{j+1})$ possibilities for attaching $I_{j+1}$ to $T_j$, there are now $2\times \mu_{\barcode}(I_{j+1})$ possibilities, since we may attach $I_{j+1}$ to the left or to the right of every interval appearing in $C_{\barcode}(I_{j+1})$.
\end{proof}

\begin{ex}
Consider the barcode
\[
\barcode=\left\{[1,\infty), [2,7), [3,6), [4,5)\right\}.
\]
Corollary~\ref{cor:counting-CMTs} predicts that there are $2^3\times 1\times 2\times 3=48$ chiral merge trees realizing $\barcode$.
\end{ex}

\subsection{Maps of Ordered and Chiral Merge Trees}

We can endow the collection of ordered merge trees, and hence the collection of chiral merge trees, with the structure of a category by defining them to be full subcategories of the category of ordered persistent sets.
To see this, we prove an easy lemma.

\begin{lem}\label{lem:OMT-to-OPS}
Every ordered merge tree $\pi:T \to \R$ defines an ordered persistent set.
\end{lem}
\begin{proof}
Given an ordered merge tree $\pi: T \to \R$ we can take two points $x_1$ and $x_2$ in the fiber $\pi^{-1}(t)$ and order them as follows: First, we consider their least common ancestor $v\in T$; this is the unique vertex with the lowest function value that specifies a connected sub-tree of $T$ that contains $x_1$ and $x_2$. 
The least common ancestor can alternatively be viewed as the lowest intersection point of the up-sets generated by $x_1$ and $x_2$ when one regards a merge tree as a persistent set and uses the partial order described in Definition~\ref{defn:P_S}.
Second, we consider the unique shortest paths $\gamma_1$ and $\gamma_2$ connecting $x_1$ and $x_2$ to $v$, respectively.
Since these paths necessarily pass through distinct children $w_1$ and $w_2$ of $v$ we use the total order on these children of $v$ to order $x_1$ and $x_2$.
To finish the proof, it is easy to see that if $t< t'$, then the map that takes points in the fiber over $t$ to points in the fiber over $t'$ is also order preserving: if $x_1$ and $x_2$ have a common image in the fiber over $t'$, then we're done and if not, then the path joining their images to their least common ancestor $v$ is a restriction of the paths from $x_1$ and $x_2$ to $v$.
\end{proof}

We can now define a map of ordered merge trees to be a natural transformation between the associated ordered persistent sets, as given by Lemma~\ref{lem:OMT-to-OPS}.

\begin{rmk}[Metrics for OMTs]\label{rmk:metrics-for-OMTs}
Moreover, we can use the notion of \define{interleavings} to define an extended pseudo-metric on ordered merge trees, see~\cite{morozov2013interleaving} for the application of the interleaving distance to (un-ordered) merge trees, \cite{bubenik2015metrics} for the simplest construction of the interleaving distance for persistence modules valued in arbitrary categories $\cat$, and~\cite{de2018theory} for a more subtle variant of the interleaving distance.
It should be noted that defining metrics using ordered or labeled merge trees is not entirely new and was developed in part by Elizabeth Munch and Anastasios Stefanou in~\cite{Munch2018}.
\end{rmk}


\section{The Persistence Map for Functions on the Interval}

Using the theory already developed, we can now characterize the fiber of the persistence map for suitably nice functions on the interval.
Obviously, the fiber of the persistence map
\[
PH_0 : \qquad f:[0,1] \to \R  \qquad \rightsquigarrow \qquad \barcode(F_0) 
\]
is uncountable.
However, if we introduce an equivalence relation called \define{graph-equivalence} and impose boundary conditions on the functions, then the fiber of the persistence map becomes finite and is indexed by chiral merge trees; see Figure~\ref{fig:ge-cmt}.

\begin{defn}\label{defn:equiv}
We say two continuous functions $f,g : [0,1] \to \R$ are \textbf{graph-equivalent} if there is an orientation preserving homeomorphism $\phi: [0,1] \to [0,1]$ such that $f=g \circ\phi$.
\end{defn}

\begin{figure}[h]
\centering
\includegraphics[width=.9\textwidth]{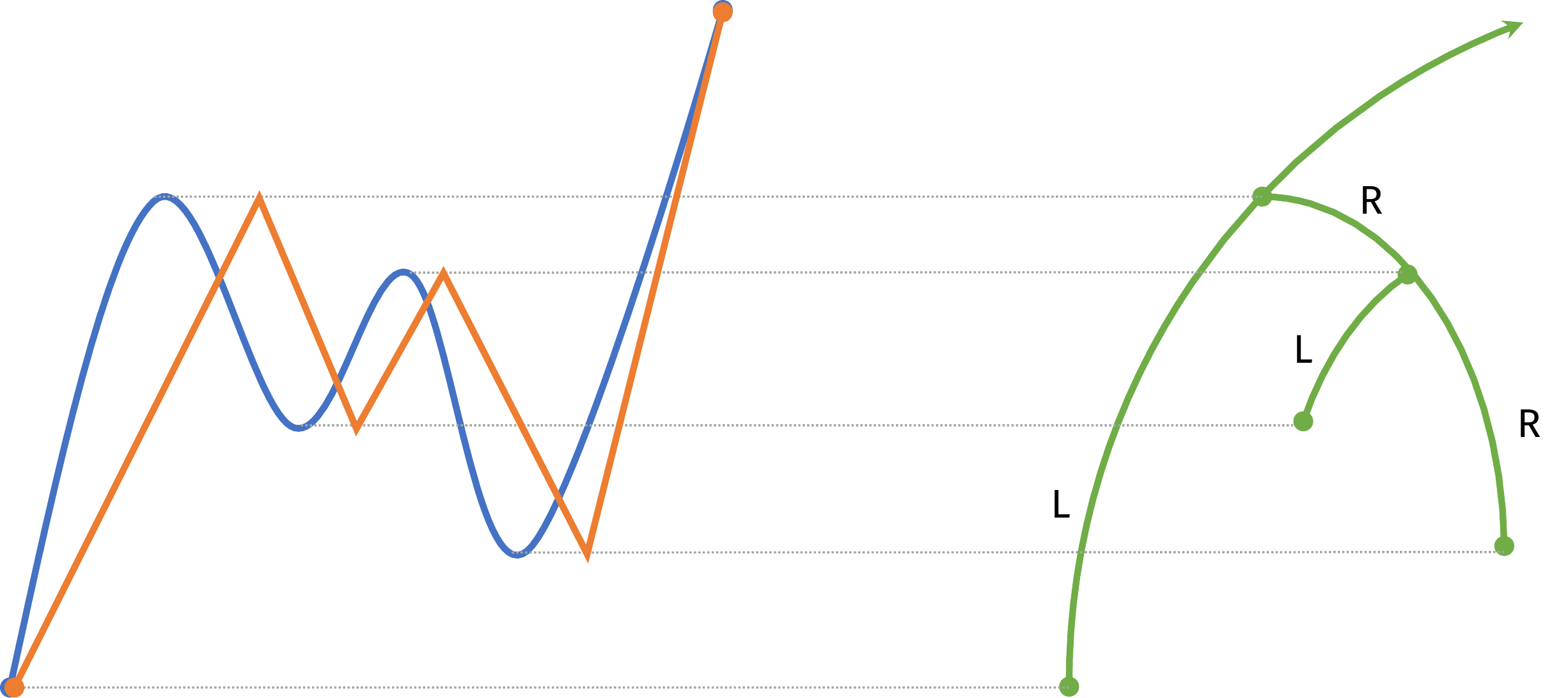}
\caption{Two Graph Equivalent Functions and their CMT}
\label{fig:ge-cmt}
\end{figure}

As the name suggests, if two functions are graph-equivalent then their graphs
$\Gamma_f = \{(x,f(x)) \,|\, x\in [0,1] \}$
and
$\Gamma_g = \{(x',g(x')) \,|\, x'\in [0,1] \}$
are homeomorphic in a level-set and orientation-preserving way.
To see this, note that 
$$\phi \times \id_{\R} (x,f(x)) = (\phi(x),f(x)) = (\phi(x),g(\phi(x))).$$
This implies that the sub-level sets of $f$ and $g$ are homeomorphic for every $t$, so in particular the persistent vector spaces
\[
	F(t):=H_0(f^{-1}(-\infty,t]) \qquad \text{and} \qquad G(t):=H_0(g^{-1}(-\infty,t])
\]
are isomorphic, so this equivalence relation is constant on fibers of the persistence map.
However, the fact that $\phi$ is orientation-preserving also makes the following true.

\begin{cor}\label{cor:ge-cmt}
If two functions $f,g: [0,1] \to \R$ are graph-equivalent, then they have isomorphic chiral merge trees.
\end{cor}

\begin{rmk}
The reverse direction of Corollary~\ref{cor:ge-cmt} is not true.
For an easy example, consider the PL function drawn in Figure~\ref{fig:ge-cmt} and consider a modification where the function is truncated at the right-most local minimum.
The original PL function and its truncation will have the same chiral merge tree, but will not be graph-equivalent because the function value at $x=0$ and $x=1$ is a graph-equivalence invariant.
\end{rmk}

Corollary~\ref{cor:ge-cmt} is one direction in a bijection that connects graph-equivalence classes of functions with chiral merge trees.
In order to go the other way, from a chiral merge tree (CMT) to a function, we must delve deeper into the structure of CMTs by proving several technical lemmas.

\begin{lem}\label{lem:vertex-order}
A chiral merge tree has a naturally associated total order on its vertices.
\end{lem}
\begin{proof}
As noted in Definition~\ref{defn:merge-tree} a CMT is rooted at the vertex $v_{\infty}$.
Consequently, for any vertex $v$ in $T$ there is a unique shortest path $\gamma$ from $v_{\infty}$ to $v$.
We can represent the path $\gamma$ as either a string of vertices 
$v_{\infty}v_1 \cdots v_{d}$ 
or a string of edges 
$e_1 \cdots e_d$.
Since every edge $e_i$ is labeled with an element of the set $\{L, R\}$, we can associate to $\gamma$ a unique sequence of letters $x_1 \cdots x_{d}$ where each $x_i$ indicates whether $e_i$ is a left incoming edge ($L$) or a right incoming edge ($R$) to $v_{i-1}$
Note that the sequence $x_1\cdots x_{d}$ also uniquely determines the path $\gamma$ and hence the vertex $v$, with the number $d$ indicating the depth of $v$ in the tree.

Suppose we have a vertex $v$, represented by a $L/R$ sequence $x_1\cdots x_d$, and another vertex $w$, represented by a $L/R$ sequence $y_1\cdots y_{d'}$.
Since in general $d\neq d'$, we append a sequence of empty characters $\varnothing$ to whatever string is shortest. 
By using the rule that $L < \varnothing < R$ we can use the lexicographical ordered to order $v$ and $w$.
\end{proof}

\begin{ex}
As an example, suppose we have three vertices, represented by $LL$, $L$ and $LR$.
The ordering described would put $LL < L\varnothing < LR$.
\end{ex}

\begin{rmk}
For a general ordered merge tree it's not clear how to specify a total ordering on the vertices. 
In particular, at a vertex with an odd number of children, it is not clear how to order this vertex with respect to its children. 
More concretely, this question amounts to a choice between orderings of the form $1 < \varnothing < 2 < 3$ and $1 < 2 < \varnothing < 3$, neither of which is canonical.
\end{rmk}

One of the properties of this ordering is that leaf nodes are exactly the odd numbers between $1$ and $N$.

\begin{lem}\label{lem:odd-leaves}
Let $T$ be a chiral merge tree.
Suppose $N$ is the number of vertices in $T$.
In the total ordering $\{v_1 < \ldots < v_N\}$ of the vertices provided by Lemma~\ref{lem:vertex-order}, the leaf nodes correspond to $v_i$ when $i$ is odd.
\end{lem}

\begin{proof}
We use a recursive description of the enumeration given in Lemma~\ref{lem:vertex-order}.
This is the \texttt{IN-ORDER} tree traversal common in computer science, see Chapter 12 of~\cite{cormen2009introduction}.
Our input is a full binary tree with every vertex labeled with a name.
The algorithm takes in a binary tree $T$ with a distinguished vertex $v$.
If $v$ has children, then we call the algorithm again on the left sub-tree.
Print the name of $v$.
Call the algorithm on the right sub-tree.

Note that the algorithm does not stop the recursive call until $v$ has no children, i.e.~it is a leaf node.
Thus the first node name that is printed is a leaf node.
Popping out of this first, deepest level of recursion, we must print the name of the parent, which is the second name printed.
The assumption that the binary tree is full guarantees that the next call on the right subtree is not empty, so another leaf node is printed next.
This implies that leaf nodes are always printed at odd numbers.

We now explain why the order of names called in the algorithm above gives the enumeration described in Lemma~\ref{lem:vertex-order}. 
Suppose $\underline{S}$ is a string of $L/R$'s and $\underline{S}L$ is the string describing the leaf node just printed by the algorithm above. 
In the lexicographical ordering the immediate successor of $\underline{S}L$ is $\underline{S}$. 
The immediate successor of $\underline{S}$ is $\underline{S}RL\cdots L$, where the number of $L$'s appearing to the right of $R$ is maximal for the given tree, thereby implying in the algorithm above that this vertex's name is the next to be printed.
\end{proof}

We can now prove a reconstruction result.

\begin{prop}\label{prop:cmt-to-pl}
To every chiral merge tree $\pi:T \to \R$ with at least three vertices there is a PL function $f_T: [0,1] \to \R$ whose chiral merge tree is $T$.
\end{prop}
\begin{proof}
Note that the number of nodes in a chiral merge tree must be odd.
We apply Lemma~\ref{lem:vertex-order} to obtain an ordering of the vertices, which we label as $v_1 < \ldots < v_N$. 
By definition of a CMT, we also have real values $\pi(v_1), \ldots, \pi(v_N)$ associated to each of the vertices.
To each vertex we can associate a pair of coordinates $(x_i,y_i)$ with $x_i=\frac{i-1}{N-1}$ and $y_i=\pi(v_i)$.
Now connect each $(x_i,y_i)$ to $(x_{i+1},y_{i+1})$ with a straight line.
This defines a PL function $f_T:[0,1] \to \R$.

We check that the chiral merge tree of $f_T$ is $T$. 
As the proof of Lemma~\ref{lem:odd-leaves} shows, when $i$ is odd $(x_i,y_i)$ represents a leaf node and $(x_{i+1},y_{i+1})$ represents an ancestral (parent, grandparent...) node, so $y_i < y_{i+1}$.
This implies that the line connecting $(x_i,y_i)$ to $(x_{i+1},y_{i+1})$ has positive slope when $i$ is odd and negative slope when $i$ is even.
As a consequence, when $i$ is odd, the point $(x_i,y_i)$ is a local minimum and the point $(x_{i+1},y_{i+1})$ is a local maximum of the function $f_T$.
Note that the $x$-value of each of these maxima and minima are ordered by the usual order on $[0,1]$.
Let $v_i=(x_i,y_i)$ denote the $i$th vertex of the merge tree associated to $f_T$.
Note that the vertices $(x_i,y_i)$ have the same $x$-order (in-order) and $y$-values as the abstract chiral merge tree $T$.
This implies that the merge tree determined by $f_T$ is the same as $T$.
\end{proof}

\begin{defn}\label{defn:Morse-like}
We say a continuous function $f:[0,1] \to \R$ is \define{Morse-like} if it is graph-equivalent to a PL function $f_{PL}$ where every critical point is isolated and has a distinct critical value.
In other words the graph of $f_{PL}$ consists of a finite number of line segments, each of which have non-zero slope. 
\end{defn}

Although the converse of Corollary~\ref{cor:ge-cmt} is not true, by imposing boundary conditions we can prove a restricted version of the converse.

\begin{lem}\label{lem:cmt-ge}
Suppose $f$ and $g$ are two Morse-like functions having both $x=0$ and $x=1$ as local minima.
If $f$ and $g$ have isomorphic chiral merge trees, then $f$ and $g$ are graph equivalent.
\end{lem}
\begin{proof}
Since the functions $f$ and $g$ are Morse-like and since graph-equivalence is transitive and preserves CMTs, it suffices to consider the case where $f$ and $g$ are PL functions.
Let $\pi: T \to \R$ denote the common chiral merge tree of $f$ and $g$
and let $h_T: [0,1] \to \R$ denote the PL function associated to $T$ constructed in Proposition~\ref{prop:cmt-to-pl}.
We now show that $h_T$ is graph-equivalent to $f$, which, by arguing symmetrically, must be graph-equivalent to $g$.
First we note that both $f$ and $h_T$ have $N$ critical points, which we label $\{p_1 < \ldots p_N\}$ and $\{x_1 < \ldots < x_N\}$.
The only way for this not to be the case is if $f$ had a local max at either $x=0$ or $x=1$, which we ruled out by hypothesis.
Consequently, each $p_i$ and $x_i$ are the same type of critical point with the same critical value.
It is thus easy to define an affine map $\phi_i: [p_i,p_{i+1}] \to [x_i,x_{i+1}]$ making $f|_{[p_i,p_{i+1}]}=h_T \circ \phi_i$.
Concatenating these affine maps together and using the fact that $f(0)=h_{T}(0)$ and $f(1)=h_T(1)$ we can define a PL homeomorphism $\phi:[0,1] \to [0,1]$ take $f$ to $h_T$.
This proves that $f$ and $h_T$ are graph-equivalent.
Arguing with $g$ in place of $f$ proves that $g$ and $h_T$ are graph-equivalent.
Transitivity of graph-equivalence proves that $f$ and $g$ are graph-equivalent.
\end{proof}

\begin{cor}\label{cor:M-to-CMT}
Let $\mathcal{M}$ denote the set of Morse-like functions on the interval with local minima at $x=0$ and $x=1$, modulo graph-equivalence.
Let $\CMTspace$ denote the set of isomorphism classes of chiral merge trees. 
The map
\[
\Psi: \mathcal{M} \to \CMTspace
\]
is a bijection onto its image.
\end{cor}
\begin{proof}
Corollary~\ref{cor:ge-cmt} implies that the map $\Psi$ is well-defined, i.e.~that the chiral merge tree construction is invariant under graph-equivalence.
Lemma~\ref{lem:cmt-ge} then shows that $\Psi([f])=\Psi([g])$ implies that $[f]=[g]$, thereby proving that $\Psi$ is an injection.
\end{proof}

We have now reached our main result.

\begin{thm}\label{thm:count-persistence}
Let $\mathcal{M}$ and $\CMTspace$ be the sets described in Corollary~\ref{cor:M-to-CMT}.
If $\barcode=\{I_j\}_{j=1}^N$ is a barcode where $I_1=[b_1,\infty)$, $I_j=[b_j,d_j)\subset I_1$ with $d_j > d_{j+1}$ for all $j\geq 2$, and where every left-hand endpoint $b_j$ is distinct, then the number of graph-equivalence classes of functions in $\mathcal{M}$ realizing $\barcode$ is
\[
| PH_0^{-1}(\barcode)| = 2^{N-1}\prod_{j=2}^N\mu_{\barcode}(I_j).
\]
Here, as in Theorem~\ref{thm:counting-merge-trees} and Corollary~\ref{cor:counting-CMTs}, $\mu_{\barcode}(I_j):=|\{I_k\in\barcode \mid I_j \subset I_k\}|$.
Also, we have identified $PH_0$ with the composition of $\Psi:\mathcal{M} \to \CMTspace$ followed by the Chiral Elder map $\Xi_{\CMTspace}:\CMTspace \to \BCspace$.
\end{thm}
\begin{proof}
Clearly $\barcode$ satisfies the hypotheses of Corollary~\ref{cor:counting-CMTs}.
This implies that $\barcode$ is in the image of the map $\Xi_{\CMTspace}:\CMTspace \to \BCspace$, thereby allowing us to construct $| \Xi_{\CMTspace}^{-1}(\barcode)|$ many chiral merge trees realizing $\barcode$.
Lemma~\ref{lem:cmt-ge} then implies that each chiral merge tree picks out a unique graph equivalence class.
\end{proof}

\newpage

\section{Conflict of Interest Statement}

The author states that there is no conflict of interest.

\bibliographystyle{plain}
\bibliography{refs}

\end{document}